\documentclass[12pt,final]{article}
\usepackage{a4}
\usepackage{amsmath}%
\usepackage{amstext}%
\usepackage{amssymb}%
\usepackage{amsxtra}
\usepackage{graphicx}
\usepackage{showkeys}%
\usepackage{epsfig}%
\usepackage{cite}
\usepackage{tikz}
\usepackage{caption}
\usepackage{multirow}
\usepackage{booktabs}
\usepackage{floatrow}
\usepackage{amsmath}
\usepackage{amssymb}
\usepackage{graphicx}
\usepackage[colorlinks,urlcolor=red]{hyperref}
\usepackage{graphicx}
\usepackage{subfigure}% Support for small, `sub' figures and tables
\usepackage{listings}
\usepackage[ruled,vlined]{algorithm2e}

\usepackage{geometry}
\newtheorem{theorem}{Theorem}[section]
\newtheorem{proposition}[theorem]{Proposition}

\newtheorem{remark}[theorem]{Remark}

\newcounter{unnumber}

\newtheorem{prf}[unnumber]{Proof.}
\newenvironment{proof}{\prf\rm}{\hfill{$\square$}\endprf}
\numberwithin{equation}{section} 

\newcommand{\R}{\mathbb{R}}%
\newcommand{\N}{\mathbb{N}}%
\DeclareMathOperator*\argmin{argmin}

\DeclareMathOperator*\fix{Fix}

\title{\bf\Large Star subgradient projection for solving quasi-convex feasibility problems}

\author{Nimit Nimana$^1$\and
	Narin Petrot$^2$}

\date{\small 
	$^1$Department of Mathematics, Faculty of Science, Khon Kaen University,
	Khon Kaen,  Thailand\\
	$^2$Department of Mathematics, Faculty of Science, Naresuan University,
	Phitsanulok,  Thailand}
\begin{document}
	\maketitle	
	\noindent \textbf{Abstract.} In this work we consider an iterative method for solving the quasi-convex feasibility problem.  We firstly introduce  the so-called star subgradient projection operator and present some useful properties. We subsequently obtain a convergence result of the considered method by using properties of the introduced nonlinear operator.
		\vspace{1ex}
		
	\vspace{1ex}
	
		\vspace{1ex}

	%%%%%%%%%%%%%%%%%%%%%%%%%%%%%%%%%%%%%%%%%%%%%%%%%%%%%%%%%%%%%%%%%%%%%%%%%%%%%%%%

% ---------------------Body---------------------------------

\section{Introduction}
\hskip0.6cm

Let $f_i:\R^n\to\R$ be convex representative functions, $i=1,\ldots,m$, the convex feasibility problem is to find a point $x^*\in\R^n$ satisfying $x^*\in S^{f_i}_{\leq,0},  i=1,\ldots,m,$
where $S^{f_i}_{\leq,0}:=\{x\in\R^n:f(x)\leq0\}$ is the zero sublevel set of $f_i$ corresponding to the level $0$, and provided that the intersection is nonempty. It is well known that the convex feasibility problem plays an important role in the modellings of many noticeable situations, for example, signal processing, image processing, sensor network localization problems, for more information, see \cite{Blatt-Hero06,CCG12,Combettes96} and references therein.  To deal with the convex feasibility problems, one often utilizes the so-called subgradient projection operator corresponding to each function $f_i$. Actually, we know that the subgradient projection is a cutter with its fixed point set equals to the zero sublevel set of the considered function $f_i$, furthermore, it is satisfying the so-called fixed point closed property. In this situation, the convergence results of  methods for solving the convex feasibility problems can be obtained by applying the convergence results of the cutter operator.  For more details about convex feasibility problems, algorithms and convergence properties, we refer to \cite{Bauschke-Borwein96,C12}.

Even if the convexity of the representative function has been studied and applied to several aspects, there are some situations such that the representative function is not convex, for instance in economics \cite{BF74,S97}, but satisfying the so-called quasi-convexity. The formal known property of quasi-convex is its sublevel set is a convex set. Of course, in a similar fashion to the convex feasibility problem, many authors also consider the so-called quasi-convex feasibility problems. Their solving iterative methods and convergence results can be found in, for instance  \cite{Censor-Segal06,NAP16}.

In this paper we also deal with algorithmic properties of a method for solving the quasi-convex feasibility problem. We firstly introduce a nonlinear  operator corresponding to a quasi-convex function. Under some suitable assumptions, we show some important properties of the introduced operator. Finally, we show the convergence of the introduced iterative method.
%%%%%%%%%%%%%%%%%%%%%%%%%%%%%%%%%%%%%%%%%%%%%%%%%%%%%%%%%%%%%%%%%%%%%%
\section{Preliminaries}
\hskip0.6cm

Let $\R^n$ be a Euclidean space with an inner product $\langle \cdot , \cdot \rangle$ and with the norm $\parallel  \cdot \parallel$. Let $f:\R^n\to\R$ be a function and $\lambda$ be a real number. The {\it strictly sublevel} and {\it sublevel}  sets of $f$ corresponding to $\lambda$ are defined by $S^f_{<,\lambda}:=\{x\in \R^n:f(x)<\lambda\}$ and  $S^f_{\leq,\lambda}:=\{x\in \R^n:f(x)\leq\lambda\}$, respectively.
For a set $A$, we denote by $\mathrm{cl}(A)$ its closure.
Note that $S^f_{\leq,\lambda}=\mathrm{cl}(S^f_{<,\lambda})$ may fail in general. 

A function $f:\R^n\to\R$ is said to be \textit{upper semiconinuous} on $\R^n$ if
$S^f_{<,\lambda}$  is an open set for all $\lambda\in\R$.

As we know that if $f$ is convex, the sets  $S^f_{\leq,\lambda}$ and $S^f_{<,\lambda}$ are convex for every $\lambda\in\R$. However, the converse is generally false.   A function $f:\R^n\to\R$ is said to be \textit{quasi-convex} on $\R^n$ if $ S^f_{\leq,\lambda}$ (also, $S^f_{<,\lambda}$) is a convex set for all $\lambda\in\R$.

Now we are going to recall some generalized subdifferentials and their important properties which are needed in the sequel.  
In 1973, Greenberg and Pierskalla \cite{GP73} introduced the  so-called Greenberg-Pierskalla subdifferential.
Let $f:\R^n\to\R$ be a function and $x\in\R^n$. An element $g\in\R^n$ is a {\it Greenberg-Pierskalla subgradient} of $f$ at $x$ if 
\begin{eqnarray*}
	\langle g,y-x\rangle<0,\hspace{0.15cm} \textrm{ for every } y\in S^f_{<,f(x)}.
\end{eqnarray*}
We call the set of all Greenberg-Pierskalla subgradients of $f$ at $x$ the {\it Greenberg-Pierskalla subdifferential} of $f$ at $x$, and will be denoted by $\partial^{\mathrm{GP}}f(x)$. 
It is clear by the definition of $S^f_{<,f(x)}$ that $\partial^{\mathrm{GP}}f(x)=\R^n$ whenever $S^f_{<,f(x)}=\emptyset$.

Note that for any $x\in \R^n$, the set  $\partial^{\mathrm{GP}}f(x)$ may not be closed in general.
To overcome this drawback, we consider the following definition introduced by Penot \cite{Penot98} and further investigated by Penot and Z\u{a}linescu \cite{PZ00}. 
Let $f:\R^n\to \R$ be a function and $x\in\R^n$. An element $g\in\R^n$ is the \textit{star subgradient} of $f$ at $x\in \R^n$ if
\begin{eqnarray*}\langle g,y-x\rangle\leq0,\hspace{0.15cm}  \textrm{ for all } y\in S^f_{<,f(x)}.
\end{eqnarray*}
The set of all star subgradients of $f$ at $x$ is called the {\it star subdifferential} of $f$ at $x$ and it is denoted by $ \partial^{\star}f(x)$.

The following theorem shows some basic properties of star subdifferential. For more details, see \cite[Proposition 29-30]{Penot98}.
\begin{theorem}\label{Generalized-starsub-properties}Let $f:\R^n\to\R$ be a function and $x\in\R^n$. Then the following statements are true: 
	\begin{itemize}
		\item[(i)] $\partial^{\star}f(x)$ is a closed convex cone.
		\item[(ii)]  $\partial^{\mathrm{GP}}f(x)\subset\partial^{\star}f(x)$.
		\item[(iii)] $0\in\partial^{\star}f(x)$ if and only if $x\in\argmin f$.
	\end{itemize}
\end{theorem}

The  nontrivialness of the star subdifferential is guaranteed  by the following theorem appeared in \cite[Proposition 31]{Penot98}. 
\begin{theorem}\label{starnonempty}Let  a function $f:\R^n\to \R$ be quasi-convex and upper semicontinuous and let $x\in \R^n$ be given. Then  $\partial^{\star}f(x)\setminus\{0\}\neq\emptyset$.
\end{theorem}

%\begin{rem}
%	It follows from Theorem 3.4 of \cite{R91} that $ \partial^{GP}f(x)\setminus\{0\}\neq\emptyset$.
%\end{rem}

A  function $f:\R^n\to\R$ with $S^f_{\leq,0}\neq\emptyset$ is said to be satisfying {\it property (\textbf{sH\"{o}l})} on $ S^f_{\leq,0}$ \cite{NAP16} if there exist $\delta>0$ and $L>0$ such that
\begin{eqnarray*}|f(x)-f(q)|\leq L\|x-q\|^{\delta},\hspace{0.15cm} \textrm{for all }q\in  S^f_{\leq,0}, x\in \R^n.
\end{eqnarray*}

We denote the positive part of a function $f$ by $f_+$, i.e., $f_+(x):=\max\{f(x),0\}$ for all $x\in \R^n$.
The following technical lemma  will play a crucial role in the sequel and its proof is due to Konnov \cite{Konnov03}.

\begin{theorem}\label{SQFP-lemmak}Let  $f:\R^n\to\R$ be a quasi-convex upper semicontinuous function with $S^f_{<,0}\neq\emptyset$. If the function $f$ satisfies the property (\textbf{sH\"{o}l}) on $ S^f_{\leq,0}$ with order $\delta$ and modulus $L$, then for each $x\notin S^f_{\leq,0}$, we have $f_+(x)\leq L\left\langle \frac{c}{\|c\|}, x-q\right\rangle^{\delta}$, for all $q\in  S^f_{\leq,0}$ and $c\in\partial^{\star}f(x)\setminus\{0\}$.
\end{theorem}

A function $f:\R^n\to \R$ is said to be $0$-\textit{lower semicontinuous}\cite{NAP16} if its zero sublevel set $ S^f_{\leq,0}$ is a closed set.
Consider a function $f:\R\to \R$ defined by
\begin{eqnarray*}f(x)=\left\{%
	\begin{array}{ll}
		\lfloor x\rfloor& \hbox{ ; $x>1$,}\\
		x& \hbox{ ; otherwise},
	\end{array}
	\right.\end{eqnarray*}
where $\lfloor x\rfloor$  is a floor function, proposed in \cite{NAP16}. Observe that $f$ is $0$-lower semicontinuous and upper semicontinuous but not lower semicontinuous.

We will close this section by recalling the concept of set convergence, which is known as Painlev\'{e}-Kuratowski
convergence. All of these definitions and some more further properties can be found in \cite[Chapter 4]{RW98} and \cite[Chapter 2]{Burachik-Iusem08}.

We denote the family of subsets of $\N$ representing all the tails of $\N$ by  
\begin{eqnarray*}\N_{\infty}:=\{N\subset \N:\N\setminus N \textrm{ is finite}\},
\end{eqnarray*}
and the family of subsets of $\N$ representing all the subsequence of $\N$ by 
\begin{eqnarray*}\N_{\infty}^\sharp:=\{N\subset \N:N \textrm{ is infinite}\}.
\end{eqnarray*}
By using these notations, the subsequence of a sequence $\{x_k\}_{k\in\N}$ has the form $\{x_k\}_{k\in N}$ with $N\in \N_{\infty}^\sharp$, while the tail of $\{x_k\}_{k\in\N}$ has the form $\{x_k\}_{k\in N}$ with $N\in \N_{\infty}$. We use the notation $\lim_{k\in N}x_k$ in the case of convergence of a subsequence of $\{x_k\}_{k\in\N}$ designated by an index set $N$ in $\N_{\infty}$ or $\N_{\infty}^\sharp$.

Let  $\{C_k\}_{k\in\N}$ be a sequence of subsets of $\R^n$ and $C\subset\R^n$. The {\it outer limit} is the set
\begin{eqnarray*}\mathrm{Lim sup}_{k\to+\infty}C_k&:=&\{x\in \R^n:\exists N\in\N_{\infty}^\sharp, \forall k\in N, \exists x_k\in C_k \textrm{ such that } \lim_{k\in N}x_k=x \}
\end{eqnarray*}
and the {\it inner limit} is the set
\begin{eqnarray*}\mathrm{Lim inf}_{k\to+\infty}C_k:=\{x\in \R^n:\exists N\in\N_{\infty}, \forall k\in N, \exists x_k\in C_k \textrm{ such that } \lim_{k\in N}x_k=x \}.
\end{eqnarray*}
We say that the sequence $\{C_k\}_{k\in\N}$ converges to $C$ if the outer and inner limit sets are equal to $C$, i.e.,
\begin{eqnarray*}\mathrm{Lim}_{k\to+\infty}C_k:=\mathrm{Lim sup}_{k\to+\infty}C_k=\mathrm{Lim inf}_{k\to+\infty}C_k=C.
\end{eqnarray*}

The following theorem will be a key tool in our work and the proof can be found in \cite[Exercise 2.2]{Burachik-Iusem08}.
\begin{theorem}\label{set-convergence-nested}Let $\{C_k\}_{k\in\N}$ be a sequence of subsets of $\R^n$ such that $C_{k+1}\subset C_{k}$ for all $k\geq1$. Then $\mathrm{Lim}_{k\to+\infty}C_k$ exists and $\mathrm{Lim}_{k\to+\infty}C_k=\bigcap_{k\in\N}\mathrm{cl}\left(C_k\right)$.
\end{theorem}

Let us denote the distance (function) in  $\R^n$ by $\mathrm{dist}:\R^n\times \R^n\to\R$ and recall that for $C\subset \R^n$,
\begin{eqnarray*}\mathrm{dist}(x,C):= \inf_{c\in C}\|x-c\|.
\end{eqnarray*}

The following theorem provides a relation between set convergence and the distance function, see \cite[Proposition 2.2.11]{Burachik-Iusem08} for more details.
\begin{theorem}\label{set-convergence-distance}Let  $\{C_k\}_{k\in\N}$ be a sequence of subsets of $\R^n$ and  $C$ be a closed subset of $\R^n$. Then, it holds that
	\begin{eqnarray*}
		\mathrm{Lim}_{k\to+\infty}C_k=C \Longleftrightarrow \lim_{k\to+\infty}\mathrm{dist}(x,C_k)=\mathrm{dist}(x,C),
	\end{eqnarray*}
	for every $x\in\R^n$.
\end{theorem}

%%%%%%%%%%%%%%%%%%%%%%%%%%%%%%%%%%%%%%%%%%%%%%%%%%%%%%%%%%%%%%%%%%%%%%%%%%%
%Main Results
\section{Star Subgradient Projection Operator}

In this section we will introduce an important operator for dealing with the quasi-convex feasibility problem.

Let $f:\R^n\to\R$ with $S^{f}_{\leq,0}\neq\emptyset$ be a quasi-convex, upper semicontinuous, $0$-lower semicontinuous, and satisfying the Property (\textbf{sH\"{o}l}) on $ S^f_{\leq,0}$ with order $\delta>0$ and modulus $L>0$. Let $c_f(x)\in\partial^{\star}f(x)$ be a nonzero star subgradient of $f$ at $x\in\R^n$. The operator $P_f:\R^n\to\R^n$ defined by 
\begin{eqnarray}\label{SSPP}P_f(x)=\left\{%
	\begin{array}{ll}
		x-\left(\frac{f_+(x)}{L}\right)^{1/\delta}\frac{c_f(x)}{\|c_f(x)\|}& \hbox{ if $f(x)>0$,}\\
		x& \hbox{   if $f(x)\leq0$,}
	\end{array}
	\right.\end{eqnarray}
is called a {\it star subgradient projection relative to} $f$

Obviously, for $x\notin S^{f}_{\leq,0}$, we have  $f(x)=f_+(x)$ and $f(x)>0\geq\inf_{u\in\R^n}f(u)$. This means that $x$ is not a minimizer of $f$ and it follows from Theorem \ref{Generalized-starsub-properties} (vi) that $0\notin \partial^{\star}f(x)$. Consequently, Theorem \ref{starnonempty} yields that there always exists a nonzero star subgradient $c_f(x)\in\partial^{\star}f(x)$. Therefore, the well-definedness of the star subgradient projection $P_f$ is guaranteed.  
\hskip0.6cm

The following proposition shows an important relation between the fixed point set of $P_f$, $$\fix P_f:=\{x\in\R^n:P_fx=x\},$$  and the sublevel set $S^f_{\leq,0}$.

\begin{proposition}\label{pp23}
	If $P_f:\R^n\to\R^n$ be a star subgradient projection relative to $f$, then 
	$$\fix P_f=S^f_{\leq,0}.$$ 
\end{proposition}
\begin{proof}
	It is clear that $S^f_{\leq,0}\subset\fix P_f$. Suppose that $x\notin S^f_{\leq,0}\neq\emptyset$. Then, $\partial^{\star}f(x)\setminus\{0\}\neq\emptyset$. In this case, we can find a nonzero star subgradient $c_f(x)\in\partial^{\star}f(x)\setminus\{0\}$ and $$\left(\frac{f_+(x)}{L}\right)^{1/\delta}\frac{c_f(x)}{\|c_f(x)\|}\neq0,$$
	consequently, $x\notin\fix P_f$. Hence, we conclude that $\fix P_f=S^f_{\leq,0}$, as required.
\end{proof}

The following proposition states an important property of the star subgradient projection operator.
\begin{proposition}
	If $S^f_{<,0}\neq\emptyset$, then $P_f$  is a cutter, that is 
	\begin{eqnarray*}
		\langle P_fx-x,P_fx-y\rangle\leq0, 
	\end{eqnarray*}
	for all $x\in\R^n$  and for all $y\in\mathrm{Fix}P_f$.
\end{proposition}
\begin{proof}
	From Proposition \ref{pp23}, we note here again that $\fix P_f=S^f_{\leq,0}$. If $x\in S^f_{\leq,0}$, then it is clear that $P_f$ is a cutter. Suppose that $x\notin S^f_{\leq,0}$ and $y\in S^f_{\leq,0}$. Then, $f(y)\leq0<f(x)$. Now, by invoking the definition of star subgradient projection and Theorem \ref{SQFP-lemmak}, we have
	\begin{eqnarray*}
		\left\langle P_fx-x,P_fx-y\right\rangle &=&\|P_fx-x\|^2+\left\langle P_fx-x,x-y\right\rangle\\
		&=&\left(\frac{f_+(x)}{L}\right)^{2/\delta} - \left(\frac{f_+(x)}{L}\right)^{1/\delta}\left\langle \frac{c_f(x)}{\|c_f(x)\|},x-y\right\rangle\\
		&\leq&\left(\frac{f_+(x)}{L}\right)^{2/\delta} - \left(\frac{f_+(x)}{L}\right)^{2/\delta}\\
		&=&0,
	\end{eqnarray*}
	which completes the proof.\end{proof}

The following proposition shows the so-called {\it fixed-point closed} property of the star subgradient projection operator.
\begin{proposition}
	If $S^f_{<,0}\neq\emptyset$, then $P_f$  is fixed-point closed, that is, 
	for any sequence $\{x_k\}_{k\in\N}\subset\R^n$ such that $x_k\to x\in\R^n$ as $k\to+\infty$ and $\lim_{k\to+\infty}\|P_fx_k-x_k\|=0$, we have $x\in\fix P_f$.
\end{proposition}
\begin{proof}
	Let $\{x_k\}_{k\in\N}\subset\R^n$ be a sequence such that $x_k\to x\in\R^n$ as $k\to+\infty$ and $\lim_{k\to+\infty}\|P_fx_k-x_k\|=0$. Note that 
	$$\left(\frac{f_+(x_k)}{L}\right)^{/\delta}=\|P_fx_k-x_k\|\to0,$$
	and then	$\lim_{k\to+\infty}f_+(x_k)=0.$
	Thus,  for each $n\in\N$,  there exists $k_n\in\N$ such that $f(x_k)\leq f_+(x_k)<\frac{1}{n}$ for all $k\geq k_n$. That is,  $x_k\in S^f_{<,\frac{1}{n}}$
	and subsequently that $$\mathrm{dist}\left(x_k,\mathrm{cl}(S^f_{<,\frac{1}{n}})\right)=0,$$ for all $k\geq k_n$.

	Since $y_{k}\to x$  as $k\to+\infty$, we also have that
	\begin{eqnarray*}\mathrm{dist}\left(x,\mathrm{cl}(S^f_{<,\frac{1}{n}})\right)=\lim_{l\to+\infty}\mathrm{dist}\left(x_{k},\mathrm{cl}(S^f_{<,\frac{1}{n}})\right)=0,\indent\textrm{for all }n\in\N.
	\end{eqnarray*}
	This implies that
	\begin{eqnarray}\label{lim1n}\lim_{n\to+\infty}\mathrm{dist}\left(x,\mathrm{cl}(S^f_{<,\frac{1}{n}})\right)=0.
	\end{eqnarray}
	
	On the other hand, since $f$ is a quasi-convex and upper semicontinuous function, we have that $S^f_{<,\frac{1}{n}}$ is a convex and open set for all $n\in\N$ and, further, we also have 
	$$\bigcap_{n\in\N}S^f_{<,\frac{1}{n}}(\supset S^f_{<,0})$$ is a nonempty convex set. Moreover, we observe that  $\{S^f_{<,\frac{1}{n}}\}_{n\in\N}$ and  $\{\mathrm{cl}(S^f_{<,\frac{1}{n}})\}_{n\in\N}$ are both decreasing. Further, we note that
	\begin{eqnarray*}\bigcap_{n\in\N}S^f_{<,\frac{1}{n}}= S^f_{\leq,0}.\end{eqnarray*}
	Thus, it follows from Theorem \ref{set-convergence-nested} and the property of closure that
	\begin{eqnarray}\label{SQFP-limsigma}
	\mathrm{Lim}_{n\to+\infty}\mathrm{cl}\left(S^f_{<,\frac{1}{n}}\right)=\bigcap_{n\in\N}\mathrm{cl}\left(S^f_{<,\frac{1}{n}}\right)
	=\mathrm{cl}\left(\bigcap_{n\in\N}S^f_{<,\frac{1}{n}}\right)=\mathrm{cl}( S^f_{\leq,0}).
	\end{eqnarray}

	Since $f$ is $0$-lower semicontinuous, we note that the sublevel set $S^f_{\leq,0}$ is a closed set. Therefore, invoking  (\ref{SQFP-limsigma}) together with Theorem \ref{set-convergence-distance}, we obtain that
	\begin{eqnarray}
	\label{SQFP-finitedist}
	\mathrm{dist}\left(x, S^f_{\leq,0}\right)=\mathrm{dist}\left(x,\mathrm{cl}(S^f_{\leq,0})\right)=\lim_{n\to+\infty}\mathrm{dist}\left(x,\mathrm{cl}(S^f_{<,\frac{1}{n}})\right)=0\end{eqnarray}
	and hence $x\in S^f_{\leq,0}=\fix P_f$. This completes the proof.\end{proof}

\section{Cyclic Star Subgradient Projection Method}
\hskip0.6cm

%\begin{problem}\label{SQFP-pb}
	Let $f_i:\R^n \rightarrow \R, i=1,\ldots,m,$ be quasi-convex, upper semicontinuous, $0$-lower semicontinuous, and satisfying the Property (\textbf{sH\"{o}l}) on $ S^{f_i}_{\leq,0}$ with order $\delta_i>0$ and modulus $L_i>0$, respectively. 
	The \textit{quasi-convex feasibility problem}  (in short, $\mathbf{QFP}$) is to
	find
	\begin{eqnarray*}\label{sqfp}\centering
		x^*\in \bigcap_{i=1}^mS^{f_i}_{\leq,0},
	\end{eqnarray*}
	provided that the intersection is nonempty.
%\end{problem}

In this section we are concerned with the study of convergence properties of an iterative algorithm which approaches a solution of the following  \textbf{QFP}. The following  iterative algorithm for solving the \textbf{QFP} is due to Censor-Segal\cite[Algorithm 13]{Censor-Segal06}.

\begin{algorithm}[H]\label{algorithm}
	\SetAlgoLined
	\textbf{Initialization}:  Take $x_1\in \R^n_1$ be arbitrary.\\
	\textbf{Iterative Step}: For a given current iterate $x_k\in \R^n$ ($n\geq 1$), calculate $y^i_k\in \R^n_1$ by
	$$y^0_k:=x_k$$
	\begin{eqnarray*}y^i_k:=y^{i-1}_k-\left(\frac{(f_i)_+(y^{i-1}_k)}{L_i}\right)^{1/\delta_i}\frac{c^{i-1}_k}{\|c^{i-1}_k\|},\indent i=1,\ldots,m,\end{eqnarray*}
	where $c^{i-1}_k\in\partial^{\star}f_i(y^{i-1}_k)$ is an arbitrary nonzero star subgradient of $f_i$ at $y^{i-1}_k$.\\
	Compute the next iterate $x_{k+1}\in\R^n$ by
	$$x_{k+1}:=y^{m}_k.$$
	Update $k:=k+1$.
	\caption{Cyclic Star Subgradient Projection Method}
\end{algorithm}

\begin{remark}\label{SQFP-remark-key}
	(i) Observe that the iterate $y_k^{i}$ in Algorithm \ref{algorithm} can be represented in the form of the star subgradient projection, that is
	$$y^{i}_k=P_{f_i}y^{i-1}_k,\indent i=1,\ldots,m,$$
	which yields that the iterate $x_{k+1}$ is in the form of $x_{k+1}=P_{f_m}P_{f_{m-1}}\cdots P_{f_2}P_{f_1}x_k.$
	
	(ii) Note that if there exists $k_0\in\N$ in which  $f_i(x_{k_0})\leq 0$ for all $i=1,\ldots,m$, then Algorithm \ref{algorithm} terminates and the iteration $x_{k_0}$ subsequently is a solution of the \textbf{QFP}. So to deal with the later convergence, we assume that Algorithm \ref{algorithm} does not terminate in any finite number of iterations $k\geq1$.
\end{remark}

In order to deal with our convergence theorem, we need to recall an important operator. We say that an
operator $T$ having a fixed point is $\rho $\textit{-strongly
	quasi-nonexpansive}, where $\rho \geq 0$, if 
\begin{equation*}
\Vert Tx-z\Vert ^{2}\leq \Vert x-z\Vert ^{2}-\rho \Vert Tx-x\Vert ^{2}\text{
	for all }x\in 
%TCIMACRO{\U{211d} }%
%BeginExpansion
\mathbb{R}
%EndExpansion
^{n}\text{ and all }z\in \fix T.
\end{equation*}%
for any $\lambda \in (0,2]$. 

Next, we will investigate convergence analysis of a sequence generated by the cyclic star subgradient projection method (Algorithm \ref{algorithm}) as the following theorem.

\begin{theorem}\label{main-thm} If the intersection $\bigcap_{i=1}^mS^{f_i}_{<,0}$ is nonempty, then any sequence $\{x_k\}_{k\in\N}$ generated by Algorithm \ref{algorithm} converges to a solution to \textbf{QFP}.
\end{theorem}
\begin{proof}Firstly, let us denote 
	\begin{equation*}
	T:=P_{f_m}P_{f_{m-1}}\cdots P_{f_2}P_{f_1}\text{,}
	\end{equation*}%
	where $P_{f_i}, i=1\ldots,m$ are defined by (\ref{SSPP}). Then, Algorithm \ref{algorithm} can be written in the form 
	\begin{equation*}
	x^{k+1}=Tx_{k}\text{.}
	\end{equation*}%
	Since  $P_{f_i}, i=1\ldots,m$ are cutters, each $P_{f_i}$ is nothing else than the $1$-SQNE \cite[Corollary 2.1.40]{C12}. Furthermore, since the intersection $\bigcap_{i=1}^mS^{f_i}_{\leq,0}\neq\emptyset$, we get that these operators have
	a common fixed point, which yields that the operator $T$ is also SQNE \cite[Theorem 2.1.48]{C12}. Moreover, $P_{f_i}$ are fixed-point closed, $i=1,\ldots,m$, the composition $T$
	is fixed-point closed \cite[Theorem 4.2]{C15}. Thus, the assumptions of \cite[Theorem 5.11.1]%
	{C12} are satisfied, and hence the sequence $x_{k}$ converges	to a point $x^{\ast }\in \bigcap_{i=1}^mS^{f_i}_{\leq,0}$.
\end{proof}

\section{Conclusion}

This paper introduced the so-called star subgradient projection operator and discussed their useful properties. We applied such operator for  solving the quasi-convex feasibility problem. In our opinion, this operator can be utilized when proving convergence result of another method likes the cyclic star subgradient methods and, moreover,  their properties should be investigated in the same way as the celebrated subgradient projection operator.

$^{1}$ (N. Nimana)\newline
\indent Department of Mathematics \newline
\indent Faculty of Science  \newline
\indent Khon Kaen University \newline
\indent Khon Kaen 40002, Thailand \newline
\indent E-mail address: nimitni@kku.ac.th

\vspace{0.5cm}
	$^{2}$ (N. Petrot)\newline
	\indent Department of Mathematics \newline
	\indent Faculty of Science  \newline
\indent 	Naresuan University \newline
	\indent Phitsanulok 65000, Thailand \newline
	\indent E-mail address: narinp@nu.ac.th

\end{document}